\documentclass[12pt,reqno]{amsart}

\usepackage{amsthm,amssymb,amstext,amscd,euscript,mathrsfs,amsfonts,amsbsy,amsxtra,latexsym,amsmath}
\usepackage{fullpage}
\usepackage[english]{babel}
\usepackage[latin1]{inputenc}
\usepackage{verbatim}
\usepackage{graphicx} 
\usepackage[all]{xy} 
\usepackage{enumitem}
\usepackage{bbm}
\usepackage{marvosym}
\usepackage{hyperref}
\numberwithin{figure}{section} 
\allowdisplaybreaks

\newcommand{\field}[1]{\mathbb{#1}} 

\newcommand{\Z}{\field{Z}}

\newcommand{\rk}{\operatorname{rank}}
\newcommand{\Pic}{\operatorname{Pic}}
\newcommand{\MW}{\operatorname{MW}}
\newcommand{\NS}{\operatorname{NS}}

\newtheorem{theorem}{Theorem}[section]
\newtheorem{prop}[theorem]{Proposition}
\newtheorem{lemma}[theorem]{Lemma}
\newtheorem{cor}[theorem]{Corollary}
\newtheorem{conj}[theorem]{Conjecture}

\theoremstyle{definition}
\newtheorem{definition}[theorem]{Definition}
\newtheorem{rem}[theorem]{Remark}
\newtheorem{example}[theorem]{Example}

\newcommand{\bea}{\begin{eqnarray}} 
\newcommand{\eea}{\end{eqnarray}} 
\newcommand{\be}{\begin{equation}} 
\newcommand{\ee}{\end{equation}} 
\newcommand{\benn}{\begin{equation*}} 
\newcommand{\eenn}{\end{equation*}}

\newcommand{\cE}{\mathscr{E}}

\newcommand{\cI}{\mathscr{I}}

\newcommand{\cL}{\mathscr{L}}

\newcommand{\cO}{\mathscr{O}}

\newcommand{\IC}{\mathbb{C}}

\newcommand{\IP}{\mathbb{P}}
\newcommand{\IQ}{\mathbb{Q}}

\newcommand{\IZ}{\mathbb{Z}}

\title[LTP for elliptic CY]{On a Lefschetz-type phenomenon for elliptic Calabi--Yaus}

\author{James Fullwood}
\address{James Fullwood\\School of Mathematical Sciences\\Shanghai Jiao Tong University\\ 800 Dongchuan Road, Shanghai, China}
\email{fullwood@sjtu.edu.cn}

\author{Andrea Cattaneo}
\address{Andrea Cattaneo\\
Dipartimento di Scienze Matematiche, Fisiche e Informatiche\\
Unit\`a di Matematica e Informatica\\
Universit\`a degli Studi di Parma\\
Parco Area delle Scienze 53/A, 43124\\
Parma, Italy}
\email{andrea.cattaneo@unipr.it}



\begin{document}

\begin{abstract}
We consider 18 families of elliptic Calabi--Yaus which arise in constructing $F$-theory compactifications of string vacua, and show in each case that the upper Hodge diamond of a crepant resolution of the associated Weierstrass model coincides with the upper Hodge diamond of the (blown up) projective bundle in which the crepant resolution is naturally embedded. Such results are unexpected, as each crepant resolution we consider does \emph{not} satisfy the hypotheses of the Lefschetz hyperplane theorem. In light of such findings, we suspect that all ellipitic Calabi--Yaus satisfy such a `Lefschetz-type phenomenon'.
\end{abstract}

\maketitle

\tableofcontents

\section{Introduction}
The Hodge numbers of a complex algebraic variety are fundamental invariants which in general are difficult to compute. In the context of string theory, the Hodge numbers of the compactified dimensions of string vacua determine the associated particle spectrum, and as such, play a significant role in determining which vacua have the potential for the realization of realistic particle physics. Since the discovery in the 1980s that Calabi--Yau manifolds yield natural geometries for the compactified dimensions of string vacua, string theorists have been particularly interested in computing the Hodge numbers of Calabi--Yau 3- and 4-folds, giving rise to a cross-fertilization between complex geometry and physics whose cultivation has yielded various branches of `mirror symmetry'. And while much is known about the Hodge numbers of low dimensional Calabi--Yaus (particularly when embedded as a complete intersection in a toric variety), there are still many classes of examples in which standard methods for computing Hodge numbers are not adequate to determine the whole Hodge diamond. In particular, elliptic Calabi--Yau 3- and 4-folds are a case of special interest, as investigating the Hodge structure of such varieties remains an active topic of current research \cite{GW11}\cite{CGP}\cite{Esole}\cite{FvH17}\cite{KLOOS}\cite{Taylor1}\cite{Taylor2}\cite{SHISH}.

In this letter, we consider elliptic Calabi--Yau varieties which arise in constructing $F$-theory compactifications of string vacua \cite{VFT}\cite{GTSGUT}\cite{KLRY}\cite{MV1}\cite{MV2}\cite{DW17}\cite{BHV1}\cite{BHV2}, which is promising candidate for realizing the standard model within the framework of string theory \cite{TT17}\cite{RTT19}\cite{AL57}\cite{CHLL1}\cite{WLL19}. A total of 18 families of elliptic Calabi--Yaus are considered, consisting of fourteen 3-fold families with simple gauge groups and four 4-fold families, two of which have non-trivial Mordell-Weil rank \cite{MMTW1}\cite{MDSP1}. For the 3-fold families we work over a base consisting of an arbitrary rational surface, while for the 4-fold families we work over a base which is an arbitrary Fano toric 3-fold (of which there are 18). Given such an elliptic Calabi--Yau $X\to B$, the choice of a section $B\hookrightarrow X$ yields a birational model $W_X\to B$ referred to as a \emph{Weierstrass model} of $X$, which is obtained by contracting the irreducible components of the singular fibers of $X\to B$ not meeting the image of the section $B\hookrightarrow X$. The Weierstrass model $W_X$ is in general singular, and is given by a global Weierstrass equation in the total space of a $\mathbb{P}^2$-bundle $\mathbb{P}(\mathscr{E})\to B$. We then consider a crepant resolution $\widetilde{W}_X\to W_X$, which is obtained by taking the proper transform of $W_X$ with respect to a birational map $\widetilde{\mathbb{P}(\mathscr{E})}\to \mathbb{P}(\mathscr{E})$ corresponding to a sequence of blowups along smooth centers. What we show in this letter is that in all cases under consideration, the upper Hodge diamond of $\widetilde{W}_X$ coincides with the upper Hodge diamond of the blown up projective bundle $\widetilde{\mathbb{P}(\mathscr{E})}$ in which $\widetilde{W}_X$ is naturally embedded. This result is unexpected, as we show in all cases that $\widetilde{W}_X$ does \emph{not} satisfy the hypotheses of the Lefschetz hyperplane theorem (which would imply such a match between the upper Hodge numbers of $\widetilde{W}_X$ and $\widetilde{\mathbb{P}(\mathscr{E})}$). While there is in fact a generalization of the Lefschetz hyperplane theorem to line bundles which are `lef' by de Cataldo and Migliorini \cite{DeCataldoMigliorini}, in \S\ref{sect: non-lef} we prove a general result showing that in such a context $\mathscr{O}(\widetilde{W}_X)$ cannot be lef.

In light of our results, we highly suspect that \emph{all} (or at least a large class of) elliptic Calabi--Yaus satisfy such a `Lefschetz-type phenomenon', or simply `LTP' for short. If the LTP conjecture is in fact true, then it would provide a tractaable avenue for computing the full Hodge diamond of a general elliptic Calabi--Yau $X\to B$. In particular, the birational invariance of the Hodge numbers of Calabi--Yau varieties implies that if $\widetilde{W}_X\to W_X$ is a crepant resolution of the Weierstrass model of $X$, then the Hodge numbers of $X$ coincide with those of $\widetilde{W}_X$. As such, computing the Hodge numbers of $X$ then reduces to computing the Hodge numbers of the crepant resolution $\widetilde{W}_X$, which is naturally embedded as a hypersurface in a projective bundle. Crepant resolutions of Weierstrass models have been studied in depth in recent years \cite{Esole}\cite{EsoleYau}\cite{Sakura}\cite{SNBG1}\cite{SNBG2}\cite{SNBG3}\cite{MSN1}\cite{KMW}, so if LTP in fact holds, then one can combine such resolution procedures together with Hirzebruch-Riemann-Roch to obtain explicit formulas for all non-trivial Hodge numbers of $X$.

As the LTP condition for an elliptic Calabi--Yau $X\to B$ is purely numerical, it is natural to surmise that if $\widetilde{W}_X\to W_X$ is a crepant resolution of the Weierstrass model of $X$, then the inclusion $\widetilde{W}_X\hookrightarrow \widetilde{\mathbb{P}(\cE)}$ induces an isomorphism of Hodge structures for $p+q<\text{dim}(X)$. However, in \S\ref{LTPCON} we prove that this more general statement is false (see Proposition~\ref{CE119}), as we show that fibrations with non-trivial Mordell-Weil groups which admit sections with torsion will always provide counterexamples. As such, we suspect that LTP may be a consequence of the inclusion $\widetilde{W}\hookrightarrow \widetilde{\mathbb{P}(\cE)}$ inducing an isomorphism of \emph{rational} Hodge structures (we formulate a precise conjecture in \S\ref{LTPCON}). Furthermore, the Calabi--Yau assumption on $X$ seems to be crucial for $\text{dim}(X)>2$, since in \S\ref{CYCON} we construct examples non-Calabi--Yau elliptic fibrations over a base of arbitrary dimension which \emph{do not} satisfy the Lefschetz-type phenomenon. 

\addtocontents{toc}{\protect\setcounter{tocdepth}{0}}

\subsection*{Notation and conventions}
Given a vector bundle $\mathscr{E}\to B$ over a variety $B$, $\mathbb{P}(\mathscr{E})\to B$ will always be taken to denote the associated projective bundle of \emph{lines} in $\mathscr{E}$, and the tautological line bundle of $\mathbb{P}(\cE)$ and its dual will be denoted by $\cO(-1)$ and $\cO(1)$ respectively. Given a line bundle $\mathscr{L}\to B$, its $m$-th tensor power will be denoted $\mathscr{L}^m\to B$. The canonical bundle of a smooth variety $X$ will be denoted $\omega_X\to X$.

\subsection*{Acknowledgments}
Andrea Cattaneo is a member of GNSAGA of INdAM. He would also like to thank the School of Mathematical Sciences at Shanghai Jiao Tong University for the kind hospitality during September 2018, where the ideas leading to this work were first established.

\addtocontents{toc}{\protect\setcounter{tocdepth}{2}}

\section{Setting the stage for LTP}
\subsection{Some preliminaries}
Let $B$ be a smooth compact complex projective variety. 

\begin{definition}
A proper, flat, surjective morphism $\pi:X\to B$ with connected fibers will be referred to as an \emph{elliptic fibration} if and only if the generic fiber of $\pi$ is a smooth curve of genus 1, and the morphism $\pi$ admits a regular section $B\hookrightarrow X$.
\end{definition}

The singular fibers of an elliptic fibration $\pi:X\to B$ reside over a closed subscheme $\Delta_X\hookrightarrow B$, referred to as the \emph{discriminant} of $\pi:X\to B$. By contracting the irreducible components of the singular fibers of $\pi:X\to B$ which do not meet the section, we obtain a birational morphism $F:X\to W_X$ such that following diagram
\[
\xymatrix{X \ar[rr]^F \ar[dr]_{\pi} & & W_X \ar[dl]^{\psi}  \\
 & B & }
\]
commutes. The fibration $\psi:W_X\to B$ is then referred to as the \emph{Weierstrass model} of $X$.

\begin{rem}
Note that in the definition of elliptic fibration the total space is not assumed to be smooth. The flatness condition ensures that an elliptic fibration has equidimensional fibres, and observe that if the total space of the fibration is smooth then the converse also holds (see, e.g., \cite[Criterion for flatness]{Nowak}).
\end{rem}

\begin{definition} Let $\pi: X \rightarrow B$ be an elliptic fibration. The \emph{fundamental line bundle} of $\pi$ is the line bundle $\mathscr{L}\to B$ given by $\cL = \left( R^1 \pi_* \cO_X \right)^{-1}$.
\end{definition}

\begin{prop}[{cf.\ \cite[Theorem 2.1]{Nakayama}}]\label{WMX97}
Let $\pi: X \rightarrow B$ be an elliptic fibration, let $\mathscr{L}\to B$ denote the fundamental line bundle of $\pi$, and let $p:\IP(\cO_B\oplus \cL^2 \oplus \cL^{3})\to B$ be the structure map. Then the Weierstrass model $W_X$ of $X$ is naturally embedded as a hypersurface in the projective bundle $\IP(\cO_B\oplus \cL^2 \oplus \cL^{3})$, given by the equation
\begin{equation}\label{Weq}
W_X:(y^2z=x^3+fxz^2+gz^3)\subset \IP(\cO_B\oplus \cL^{2} \oplus \cL^{3}),
\end{equation}
where $f$ is a section of $p^*\cL^4$, $g$ is a section of $p^*\cL^6$ and $z$ is a section of $\mathscr{O}(1)$.
\end{prop}

The Weierstrass model of an elliptic fibration is a special case of a Weierstrass fibration, whose definition we now recall.

\begin{definition}
An elliptic fibration $\psi:W\to B$ will be referred to as a \emph{Weierstrass fibration} if and only if $W$ may be given by a global Weierstrass equation
\[
W:(y^2z=x^3+fxz^2+gz^3)\subset \IP(\cO_B\oplus \cL^{2} \oplus \cL^{3}),
\]
where $\mathscr{L}$ is the fundamental line bundle of $\psi$, $f$ is a section of $p^*\mathscr{L}^4$, $g$ is a section of $p^*\mathscr{L}^6$, $z$ is a section of $\mathscr{O}(1)$ and $p:\IP(\cO_B\oplus \cL^{2} \oplus \cL^{3})\to B$ is the structure map. If $W$ is in fact an anti-canonical divisor in $\IP(\cO_B\oplus \cL^{2} \oplus \cL^{3})$, then $W$ will be referred to as an \emph{anti-canonical Weierstrass fibration}.
\end{definition}

\begin{rem} We note that while all Weierstrass models are Weierstrass fibrations, there exist Weierstrass fibrations which are not Weierstrass models of smooth elliptic fibrations.
\end{rem}

\begin{prop}\label{LSX913}
Let $\psi:W\to B$ be a a Weierstrass fibration, and let $\mathscr{L}$ denote the fundamental line bundle of $\psi$. Then the following statements hold.
\begin{enumerate}
\item\label{VLX1}
$\mathscr{O}(W)=p^*\mathscr{L}^6\otimes \mathscr{O}(3)$.
\item\label{VLX2}
$W$ is an anti-canonical Weierstrass fibration if and only if $\mathscr{L}=\omega_B^{-1}$.
\end{enumerate}
\end{prop}
 
\begin{proof}
Let $\cE = \cO_B\oplus \cL^{2} \oplus \cL^{3}$ and $p:\IP(\cE)\to B$ be the structure map. Since $g$ is a section of $p^*\mathscr{L}^6$ and $z$ is a section of $\mathscr{O}(1)$, it follows that $gz^3$ is a section of $p^*\mathscr{L}^6\otimes \mathscr{O}(3)$. And since $W$ is the zero-scheme of the section $y^2z-(x^3+fxz^2+gz^3)$, it follows that $\mathscr{O}(W)=p^*\mathscr{L}^6\otimes \mathscr{O}(3)$, which proves item \ref{VLX1}.

To prove item \ref{VLX2}, we consider two exact sequences over $\IP(\cE)$, namely, the exact sequence defining the relative cotangent bundle
\[0 \rightarrow p^* \Omega^1_B \rightarrow \Omega^1_{\IP(\cE)} \rightarrow \Omega_{\IP(\cE)|B} \rightarrow 0\]
and the relative Euler sequence (see \cite[$\S$B.5.8]{Fulton})
\[0 \rightarrow \Omega_{\IP(\cE)|B} \rightarrow \cO_{\IP(\cE)}(-1) \otimes p^* \cE^\vee \rightarrow \cO_{\IP(\cE)} \rightarrow 0.\]
We then have
\begin{equation}\label{FCCPB}
\begin{array}{rl}
\omega_{\IP(\cE)} = & \det \Omega^1_{\IP(\cE)}\\
= & p^* \omega_B \otimes \det \Omega_{\IP(\cE)|B}\\
= & p^* \omega_B \otimes \det \left( \cO_{\IP(\cE)}(-1) \otimes p^* \cE^\vee \right)\\
= & p^* \omega_B \otimes \cO_{\IP(\cE)}(-3) \otimes \det p^* \cE^\vee\\
= & \cO_{\IP(\cE)}(-3) \otimes p^* \left( \omega_B \otimes \det \cE^\vee \right)\\
= & \cO_{\IP(\cE)}(-3) \otimes p^* \left( \omega_B \otimes \cL^{-5} \right).
\end{array}
\end{equation}
It then follows from item \ref{VLX1} and the injectivity of $p^*$ that $\cO(W) = \omega_{\IP(\cE)^{-1}}$ if and only if $\cL = \omega_B^{-1}$.

\end{proof}

\begin{definition}
Let $X$ be a smooth, projective, complex algebraic variety, and suppose $p$ and $q$ are non-negative integers such that $p+q\leq \text{dim}(X)$. Then the \emph{Hodge number} $h^{p,q}(X)$ is the non-negative integer given by
\[
h^{p,q}(X)=\text{dim}\left(H^q(X,\Omega^p)\right).
\]
\end{definition}

\begin{definition}
A resolution of singularities $f:\widetilde{X}\to X$ will be referred to as \emph{crepant} if and only if $f$ is a birational map such that $K_{\widetilde{X}}=f^*K_X$.
\end{definition}

\begin{definition}
Let $X$ be a smooth, projective, complex algebraic variety. Then $X$ is said to be \emph{Calabi--Yau} if and only if the following conditions hold:
\begin{enumerate}
\item
$\omega_X=\cO_X$.
\item 
$h^{0,0}(X) = h^{\dim(X), 0}(X) = 1$.
\item
$h^{p,0}(X) = 0$ for $0<p<\text{dim}(X)$.
\end{enumerate}
\end{definition}

\begin{definition}
An elliptic fibration $\pi:X\to B$ will be referred to as an \emph{elliptic Calabi--Yau} if and only if $X$ is Calabi--Yau.
\end{definition}

\begin{prop}\label{prop: fund line bund of CY}
Let $\pi:X \to B$ be an elliptic Calabi--Yau, and let $\mathscr{L}$ denote the fundamental line bundle of $\pi$. Then the following statements hold.
\begin{enumerate}
\item\label{SLX1}
$\mathscr{L}$ is the anti-canonical bundle of $B$, i.e., $\mathscr{L}=\omega_B^{-1}$.
\item\label{SLX2}
The associated Weierstrass model $\psi: W_X\to B$ is an anti-canonical Weierstrass fibration.
\item\label{SLX3} The morphism to the Weierstrass model $f: X \rightarrow W_X$ is a crepant resolution.
\end{enumerate}
\end{prop}

\begin{proof}
Since $X$ is Calabi--Yau $\omega_X=\mathscr{O}_X$, thus $\cL^{-1} = R^1 \pi_* \cO_X = R^1 \pi_* \omega_X$. By the projection formula and \cite[Proposition 7.6]{Kollar} we then have
\[\begin{array}{rl}
R^1 \pi_* \omega_X = & R^1 \pi_* \left( \omega_{X|B} \otimes \pi^* \omega_B \right) \\
= & R^1 \pi_* \omega_{X|B} \otimes \omega_B \\
= & \cO_B \otimes \omega_B \\
= & \omega_B,
\end{array}\]
from which item \ref{SLX1} follows. Item \ref{SLX2} then follows from item \ref{SLX1} together with item \ref{VLX2} of Proposition~\ref{LSX913}. To prove item \ref{SLX3}, we first note that the dualizing sheaf $\omega_{W_X}$ of $W_X$ is such that $\omega_{W_X}=\psi^*(\omega_B \otimes \cL)$  (see, e.g., \cite[p. 409]{Nakayama}), thus
\[
\omega_{W_X} = \psi^*(\omega_B \otimes \cL) = \psi^*(\omega_B \otimes \omega_B^{-1}) = \psi^*(\mathscr{O}_B)=\cO_{W_X},
\]
where the second equality follows from item \ref{SLX1}. We then have
\[
f^* \omega_{W_X} = f^* \cO_{W_X} = \cO_X = \omega_X,
\] 
thus $f$ is crepant.
\end{proof}

\begin{rem}
In light of item \ref{SLX3} of Proposition~\ref{prop: fund line bund of CY}, the Hodge numbers of an elliptic Calabi--Yau coincide with the \emph{stringy Hodge numbers} of its Weierstrass model (as defined by Batyrev \cite{SHNB}). 
\end{rem}

\begin{prop}\label{PLX17}
Let $\pi:X\to B$ be an elliptic Calabi--Yau, and suppose $\widetilde{W}_X\to W_X$ is a crepant resolution of the Weierstrass model of $X$. Then $\widetilde{W}_X$ is Calabi--Yau. 
\end{prop}
\begin{proof}
By item \ref{VLX2} of Proposition~\ref{prop: fund line bund of CY} $W_X$ is an anti-canonical Weierstrass fibration, hence by adjunction it follows that the dualizing sheaf $\omega_{W_X}$ is trivial, thus $\omega_{\widetilde{W_X}} = \cO_{\widetilde{W_X}}$ (since the resolution $\widetilde{W}_X\to W_X$ is crepant). Moreover, since the Hodge numbers $h^{p, 0}$ are known to be birational invariants, and $X$ is birational to $\widetilde{W}_X$, it follows that $h^{p, 0}(\widetilde{W}_X)=h^{p,0}(X)$ for $p=0,...,\text{dim}(X)$, thus $\widetilde{W}_X$ is Calabi--Yau. 
\end{proof}

In 1995 Kontsevich created the theory motivic integration for the purpose of proving the following result \cite{Kontsevich}.

\begin{theorem}
\label{KLO}
If $X\to Y$ is a birational map between Calabi--Yau varieties, then the Hodge numbers of $X$ and $Y$ coincide. 
\end{theorem}

\begin{cor}\label{CX1983}
Let $\pi:X\to B$ be an elliptic Calabi--Yau, and suppose $\widetilde{W}_X\to W_X$ is a crepant resolution of the Weierstrass model of $X$. Then the Hodge numbers of $X$ and $\widetilde{W}_X$ coincide. 
\end{cor}

\begin{proof}
The statement follows immediately from Proposition~\ref{PLX17} and Theorem~\ref{KLO}.
\end{proof}

\begin{definition}[LTP for elliptic Calabi--Yaus]
\label{D91SX}
Let $\pi:X\to B$ be an elliptic Calabi--Yau, and let $\mathscr{E}=\mathscr{O}_B\oplus \omega_B^{-2}\oplus \omega_B^{-3}$. Then $X$ is said to satisfy a \emph{Lefschetz-type phenomenon}, or \emph{LTP} for short, if and only if the following conditions hold.
\begin{enumerate}
\item\label{LTP1}
There exists a birational map $\widetilde{\mathbb{P}(\mathscr{E})}\overset{\text{LTP}}\longrightarrow \mathbb{P}(\mathscr{E})$ such that the proper transform of the Weierstrass model of $X$ yields a crepant resolution $\widetilde{W}_X\to W_X$. In such a case, the map $\widetilde{\mathbb{P}(\mathscr{E})}\overset{\text{LTP}}\longrightarrow \mathbb{P}(\mathscr{E})$ will be referred to as the \emph{LTP map}.
\item\label{LTP2}
The associated map $\widetilde{W}_X\to B$ endows $\widetilde{W}_X$ with the structure of an elliptic Calabi--Yau.
\item\label{LTP3}
$h^{p,q}(\widetilde{W}_X)=h^{p,q}(\widetilde{\mathbb{P}(\cE)})$ for $p+q<\text{dim}(X)$.
\end{enumerate}
\end{definition}

\begin{rem}
By Proposition~\ref{PLX17} $\widetilde{W}_X$ is necessarily Calabi--Yau, so condition \ref{LTP2} in the definition of LTP is imposed to ensure the map $\widetilde{W}_X\to B$ endows $\widetilde{W}_X$ with the structure of an elliptic fibration. In particular, condition \ref{LTP2} ensures $\widetilde{W}_X\to B$ is flat, thus all fibers of $\widetilde{W}_X\to B$ are equidimensional.
\end{rem}

%%%%%%%%%%%%%%%%%%%%%%%%%%%%%%%%%%%%%%%%%%%%%%%%%%%%%%%%%%%%
\subsection{Hodge numbers of the base}
%%%%%%%%%%%%%%%%%%%%%%%%%%%%%%%%%%%%%%%%%%%%%%%%%%%%%%%%%%%%
In this section we prove that if $B$ is the base of an elliptic Calaibi--Yau, then $h^{p,0}(B) = 0$ for $p > 0$. In the case of elliptic Calabi--Yau threefolds it is known (see \cite[Main Theorem]{Oguiso}) that $B$ is a rational surface, and there are also results by Grassi in \cite[$\S$2 and $\S$3]{Grassi}.

\begin{lemma}\label{lemma: hdg inequality}
Let $\pi: X \rightarrow B$ be a fibration between smooth compact complex algebraic varieties \emph{(}i.e., a surjective morphism with connected fibers\emph{)}. Then $h^{0, k}(B) \leq h^{0, k}(X)$ for $k = 0, \ldots, \dim X$. 
\end{lemma}
\begin{proof}
It follows from the Leray spectral sequence that
\[H^k(X, \cO_X) = \bigoplus_{i = 0}^k H^i(B, R^{k - i}\pi_* \cO_X).\]
Considering the index $i = k$, the lemma follows from the fact that $\pi_* \cO_X = \cO_B$. 
\end{proof}

\begin{cor}\label{CSX971}
Let $\pi:X\to B$ be an elliptic Calabi--Yau. Then $h^{p, 0}(B) = 0$ for $p>0$.
\end{cor}

\begin{proof}
Since $X$ is Calabi--Yau we have $h^{p, 0}(X) = 0$ for $p = 1, \ldots, \dim(X) - 1=\dim(B)$. The statement then immediately follows from Lemma~\ref{lemma: hdg inequality}.
\end{proof}

\begin{rem}
As Fano varieties satisfy the condition $h^{p,0}(B)=0$ for $p>0$, they serve as natural bases for elliptic Calabi--Yaus.
\end{rem}

%%%%%%%%%%%%%%%%%%%%%%%%%%%%%%%%%%%%%%%%%%%%%%%%%%%%%%%%%%%%%%%%%%%%%%%%%%
\subsection{Hodge numbers of a blown up projective bundle}
%%%%%%%%%%%%%%%%%%%%%%%%%%%%%%%%%%%%%%%%%%%%%%%%%%%%%%%%%%%%%%%%%%%%%%%%%%
As we are primarily concerned with comparing the Hodge numbers of a crepant resolution $\widetilde{W}_X\to W_X$ of a Weierstrass model of an elliiptic Calabi--Yau $X\to B$ with that of the blown up projective bundle $\widetilde{\mathbb{P}(\cE)}$ in which $\widetilde{W}_X$ is naturally embedded, we now prove some results which will enable us to compute the Hodge numbers of $\widetilde{\mathbb{P}(\cE)}$.

\begin{definition}
Let $Z$ be a smooth, projective, complex algebraic variety. The \emph{Hodge--Deligne polynomial} of $Z$ is the polynomial $E_Z(u,v)\in \Z[u,v]$ given by 
\[
E_Z(u,v) = \sum_{p, q} (-1)^{p+q} h^{p, q}(Z) u^p v^q.
\]
\end{definition}

\begin{rem}
While the Hodge--Deligne polynomial is actually defined more generally for arbitrary (i.e., possibly singular) complex projective varieties in terms of mixed Hodge structures \cite{Deligne3}, we restrict to the smooth case as this is what is needed for our purposes. 
\end{rem}

\begin{example}
The Hodge-Deligne polynomial of projective space is given by
\begin{equation}\label{hdppn}
E_{\IP^n}(u,v) = 1 + uv + (uv)^2 + \cdots + (uv)^n.
\end{equation}
\end{example}

We record some well-known properties of Hodge--Deligne polynomials via the following proposition (see, e.g., \cite{Danilov})

\begin{prop}\label{PL91X}
The Hodge--Deligne polynomial satisfies the following properties:
\begin{enumerate}
    \item 
    If $X\subset Z$ is a closed subvariety with open complement $U\subset Z$, then 
\begin{equation}\label{scissor}
E_Z(u,v)=E_X(u,v)+E_U(u,v).
\end{equation}
\item
If $Z \to B$ is a Zariski locally trivial fibration with fiber $F$ then
\begin{equation}\label{fbf}
E_Z(u,v) = E_B(u,v) \cdot E_F(u,v).
\end{equation}
\item
If $\operatorname{Bl}_W Z \to Z$ denotes the blow up of $Z$ along a smooth subvariety $W$ of codimension $m + 1$, then
\begin{equation}\label{blowuphdp}
E_{\operatorname{Bl}_W Z}(u,v) = E_Z(u, v) + (uv + \cdots +(uv)^m) E_W(u,v).
\end{equation}
\end{enumerate}
\end{prop}

The following lemmas will be used repeatedly in later sections.

\begin{lemma}\label{lemma: hdg bundle}
Let $B$ be a smooth projective variety of dimension $n$, and let $ \IP(\cE) \rightarrow B$ be a $\IP^2$-bundle over $B$. Then the following statements hold.
\begin{enumerate}
\item\label{pbf}
$E_{\mathbb{P}(\cE)}(u, v)=(1+uv+(uv)^2)E_B(u,v)$;
 \item\label{NLX2}
 $h^{p, 0}(\mathbb{P}(\mathscr{E})) = h^{p, 0}(B)$ for $0 \leq p \leq n$;
 \item
 $h^{n + 1, 0}(\mathbb{P}(\mathscr{E})) = h^{n + 2, 0}(\mathbb{P}(\mathscr{E})) = 0$.
\end{enumerate}
\end{lemma}

\begin{proof}
The lemma follows directly from Proposition~\ref{PL91X}.
\end{proof}

\begin{lemma}\label{lemma: hdg blow up}
Let $Z$ be a a smooth projective variety of dimension $n$, and $W \subseteq Z$ a smooth subvariety of codimension $m \geq 2$. Then for $0 \leq p \leq n$ we have
\[
h^{p, 0}(\operatorname{Bl}_W Z) = h^{p, 0}(Z).
\]
\end{lemma}

\begin{proof}
The lemma follows directly from Proposition~\ref{PL91X}.
\end{proof}

\begin{lemma}\label{lemma: relevant hdg}
Let $\pi: X \rightarrow B$ be an elliptic Calabi--Yau, let $\mathscr{E}=\cO_B\oplus \omega_B^{-2} \oplus \omega_B^{-3}$, and suppose there exists a birational map $\widetilde{\mathbb{P}(\mathscr{E})}\to \mathbb{P}(\mathscr{E})$ corresponding to a sequence of blowups of $\mathbb{P}(\mathscr{E})$ along smooth centers such that the proper transform $\widetilde{W}_X$ of $W_X$ yields a crepant resolution $\widetilde{W}_X\to W_X$. Then
\[
h^{p, 0}(X) = h^{p, 0}(\widetilde{W_X}) = h^{p, 0}(\widetilde{\IP(\cE)}) = 0 \qquad \text{for } 1 \leq p \leq \emph{dim}(X) - 1.
\]
\end{lemma}
\begin{proof}
The varieties $X$ and $\widetilde{W_X}$ are both Calabi--Yau, so the equality $h^{p, 0}(X) = h^{p, 0}(\widetilde{W_X})$ holds by Theorem~\ref{KLO}. It follows from Lemma \ref{lemma: hdg inequality} that $h^{p, 0}(B) = 0$ for $1 \leq p \leq \text{dim}(X) - 1$. By Lemma \ref{lemma: hdg bundle} this implies that $h^{p, 0}(\cE) = 0$ for all $p$, hence that $h^{p, 0}(\widetilde{\IP(\cE)}) = 0$ for all $p$ as blowing up a smooth subvariety does not alter $h^{p,0}$ (Lemma \ref{lemma: hdg blow up}).
\end{proof}

\begin{rem}\label{RM1X}
In later sections, we will verify LTP in a number of 3- and 4-fold examples. As such, we note that from the definition of a Calabi--Yau variety and Lemma~\ref{lemma: relevant hdg}, it follows that the relevant Hodge numbers for LTP are $h^{1,1}$ in both the 3- and 4-fold cases, and also $h^{1,2}$ in the 4-fold cases. 
\end{rem}

\begin{lemma}\label{lemma: relevant hdg 4fold}
Let $\mathbb{P}\left(\mathscr{E}\right)\to B$ be a projective bundle, and suppose $\widetilde{\mathbb{P}(\mathscr{E})}\to \mathbb{P}(\mathscr{E})$ is a birational map corresponding to a sequence of $n$ blowups of $\mathbb{P}\left(\mathscr{E}\right)$ along smooth centers. Then
\begin{equation}\label{h11pb}
h^{1,1}(\widetilde{\mathbb{P}(\cE)})=h^{1,1}(B)+1+n, 
\end{equation}
and
\begin{equation}\label{h12pb}
h^{1,2}(\widetilde{\mathbb{P}(\cE)})=h^{1,2}(B)+h^{1,0}(X_0)+\cdots +h^{1,0}(X_{n-1}),
\end{equation}
where $X_{i-1}$ is the center of the $i$-th blow up.
\end{lemma}
\begin{proof}
The expressions in \eqref{h11pb} and \eqref{h12pb} follow directly from \eqref{fbf} and \eqref{blowuphdp}.
\end{proof}

%%%%%%%%%%%%%%%%%%%%%%%%%%%%%%%%%%%%%%%%%%%%%%%%%%%%%%%%%%%%%%%%%%%%%%%%%%%%%%
\subsection{A Shioda--Tate--Wazir formula for elliptic Calabi--Yaus}
%%%%%%%%%%%%%%%%%%%%%%%%%%%%%%%%%%%%%%%%%%%%%%%%%%%%%%%%%%%%%%%%%%%%%%%%%%%%%%
In this section we use the Shioda-Tate-Wazir formula for elliptic fibrations to prove a formula for the Hodge number $h^{1,1}$ of a general elliptic Calabi--Yau. Given a variety $X$, we use the notation $\rho(X)$ to denote the Picard number of $X$.

\begin{definition}
Let $\pi: X \rightarrow B$ be an elliptic fibration with $X$ smooth. The \emph{Mordell--Weil group} of $\pi$, denoted $\MW(\pi)$, is the group which can be equivalently defined either as the group of $\IC(B)$-rational points of the generic fibre of $\pi$ or as the group of rational sections of $\pi$, namely rational maps $s: B \dashrightarrow X$ such that $\pi \circ s = \operatorname{id}$ on the domain of $s$.
\end{definition}

\begin{definition}[{Fibral divisor, \cite[Definition 3.3]{Wazir}}]
Let $D$ be an effective divisor in the total space of an elliptic fibration $\pi: X \rightarrow B$. Then $D$ is said to be a \emph{fibral divisor} if and only if $\pi(D)$ is a divisor in $B$.
\end{definition}

\begin{theorem}[{Shioda--Tate--Wazir formula, cf.~\cite[Corollary 3.2]{Wazir}}]\label{STW}
Let $\pi:X\rightarrow B$ be a smooth elliptic fibration which is not birational to a product fibration. Then
\begin{equation}\label{STWF}
\rho(X) = \rho(B) + 1 + \Gamma + \rk \MW(\pi),
\end{equation}
where $\Gamma$ is the number of irreducible and reduced fibral divisors of $\pi$ not intersecting the zero-section of the fibration.
\end{theorem}

\begin{prop}\label{cor: hdg numbers of base}
If $\pi: X \rightarrow B$ is an elliptic Calabi--Yau, then $X$ is not birational over $B$ to a product fibration.
\end{prop}
\begin{proof}
Assume that there exists a birational map $f: X \dashrightarrow B \times E$ such that the following diagram commutes
\[\xymatrix{X \ar@{-->}[rr]^f \ar[dr]_\pi & & B \times E \ar[dl]^{p}\\
 & B & },
 \]
where $p$ is the projection on the first factor and $E$ is an elliptic curve. The geometric genus $h^0(B \times E, \omega_{B \times E})$ of $B \times E$ is then zero, since
\[
h^0(B \times E, \omega_{B \times E}) = h^{n, 0}(B \times E) = h^{n - 1, 0}(B) \cdot h^{1, 0}(E) = h^{n - 1, 0}(B) = 0,
\]
where $n$ denotes the dimension of $X$, and the last equality follows from Corollary~\ref{CSX971}. On the other hand, the geometric genus is a birational invariant, thus $h^0(B \times E, \omega_{B \times E}) = h^0(X, \omega_X) = 1$, an obvious contradiction. As such, $X$ cannot be birational to a product fibration.
\end{proof}

For elliptic Calabi--Yaus, the Shioda--Tate--Wazir formula may be reformulated as follows. 

\begin{prop}[Shioda--Tate--Wazir formula for elliptic Calabi--Yaus]
\label{prop: STW for CY}
Let $\pi:X\rightarrow B$ be an elliptic Calabi--Yau with $X$ of dimension $n >2$. Then
\begin{equation}\label{STWCY}
h^{1, 1}(X) = h^{1, 1}(B) + 1 + \Gamma + \rk \MW(\pi),
\end{equation}
where $\Gamma$ is the number of irreducible and reduced fibral divisors of $\pi$ not intersecting the defining section of the fibration.
\end{prop}
\begin{proof}
It follows from the Calabi--Yau assumption on $X$ together with Corollary \ref{CSX971} that
\[h^{0, 1}(X) = h^{0, 2}(X) = 0 \qquad \text{and} \qquad h^{0, 1}(B) = h^{0, 2}(B) = 0.\]
From the long exact sequence of the exponential sequence of $X$ and $B$ we then have
\[\Pic(X) = \NS(X) = H^2(X, \IZ) \qquad \text{and} \qquad \Pic(B) = \NS(B) = H^2(B, \IZ),\]
from which it follows that $\rho(X) = \rk H^2(X, \IZ) = h^{1, 1}(X)$ and similarly for $B$. The result then follows from the Shioda--Tate--Wazir formula, as $\pi:X\to B$ is not birational to a product fibration by Proposition~\ref{cor: hdg numbers of base}.
\end{proof}

\begin{rem}
Formula \eqref{STWCY} actually holds more generally, namely, under the assumptions $h^{1, 0}(X) = h^{2, 0}(X) = 0$ and that $X$ is not birational to a product fibration.
\end{rem}

\begin{rem}
The result in Proposition \ref{prop: STW for CY} is false for $n = 2$, even if the Calabi--Yau condition still holds. In fact it is well known that a K3 surface $X$ has $h^{1, 1}(X) = 20$, while $0 \leq \rho(X) \leq 20$ and all cases can occur (for elliptically fibered K3 surfaces, all the cases with $2 \leq \rho(X) \leq 20$ occur). The main difference between the case of surfaces and the higher-dimensional fibrations is that K3 surfaces are the only Calabi--Yau varieties with $h^{2, 0} \neq 0$.
\end{rem}

\begin{prop}\label{PNX1971}
Let $\pi:X\to B$ be an elliptic Calabi--Yau, let $p:\widetilde{\mathbb{P}(\mathscr{E})}\to \mathbb{P}(\mathscr{E})$ be a birational map which is obtained by a sequence of $n$ blowups along smooth centers, and suppose the proper transform $\widetilde{W}_X$ of $W_X$ with respect to the map $p$ is an elliptic Calabi--Yau. Then $h^{1,1}(\widetilde{W}_X)=h^{1,1}(\widetilde{\mathbb{P}(\mathscr{E})})$ if and only if
\begin{equation}\label{H11QX}
n=\Gamma+\emph{rank}\left(\emph{MW}(\pi)\right),
\end{equation}
where $\Gamma$ is the number of irreducible and reduced fibral divisors of $\pi$ not intersecting the defining section of the fibration $\widetilde{W}_X\to B$. In particular, if $X$ is a 3-fold and equation \eqref{H11QX} holds, then $X$ satisfies LTP, and $p:\widetilde{\mathbb{P}(\mathscr{E})}\to \mathbb{P}(\mathscr{E})$ is the associated LTP map.
\end{prop}

\begin{proof}
By equation \eqref{h11pb} we have
\[
h^{1,1}(\widetilde{\mathbb{P}(\cE)})=h^{1,1}(B)+1+n,
\]
while the Shioda--Tate--Wazir formula \eqref{STWCY} yields
\[
h^{1,1}(\widetilde{W}_X)=h^{1,1}(B)+1+\Gamma+\text{rank}(\text{MW}(\pi)),
\]
thus $h^{1,1}(\widetilde{W}_X)=h^{1,1}(\widetilde{\mathbb{P}(\mathscr{E})})$ if and only if $n=\Gamma+\text{rank}\left(\text{MW}(\pi)\right)$, as desired.
\end{proof}

\subsection{The Lefschetz property after de Cataldo and Migliorini}\label{sect: non-lef}

In \cite{DeCataldoMigliorini}, de Cataldo and Migliorini introduced the concept of \emph{lef divisors}. In particular, they proved that if $D$ is a lef divisor in a smooth complex projective variety $M$, then for $D$ and $M$ the conclusions of the Lefschetz Hyperplane Theorem hold (see \cite[Proposition 2.1.5]{DeCataldoMigliorini}). In this section, we recall the definition of lefness, and then show that in the case we are examining --- namely, where $D$ is a crepant resolution of an anti-canonical Weierstrass fibration and $M$ is the associated blown up projective bundle --- this property does not hold, at least when the base of the fibration is Fano.

\begin{definition}[{Lef divisor, cf.~\cite[Definition 2.1.3]{DeCataldoMigliorini}}]
We say that a divisor $D$ in a variety $M$ is \emph{lef} if and only if a positive multiple of $D$ is generated by its global sections, and the corresponding morphism onto the image is semi-small, i.e., the map $\varphi = \varphi_{|D|}: M \rightarrow M'$ has the property that there is no irreducible subvariety $T \subseteq M$ such that
\[2 \dim T - \dim \varphi(T) > \dim M.\]
\end{definition}

\begin{prop}\label{prop: -kz not lef}
Let $B$ be a smooth Fano variety, and let $Z = \IP(\cO_B\oplus \omega_B^{-2} \oplus \omega_B^{-3})$. Then $\omega_Z^{-n}$ is not base-point-free for $n \geq 1$. In particular, $\omega_Z^{-1}$ is not lef.
\end{prop}
\begin{proof}
Let $p: Z \rightarrow B$ denote the structure map. By equation \eqref{FCCPB}, $\omega_Z^{-1} = p^* \omega_B^{-6} \otimes \cO_Z(3)$, hence
\[\omega_Z^{-n} = p^* \omega_B^{-6n} \otimes \cO_B(3n).\]
Let $(x: y: z)$ denote the natural coordinates on the fibers of $Z$. Then $y$ is a section of $p^* \omega_B^{-3} \otimes \cO_Z(1)$ and hence the coefficient of $y^{3n}$ in an equation for a divisor $D \in |\omega_Z^{-n}|$ must be a section of $p^* \omega_B^{3n}$. However,
\[H^0(Z, p^* \omega_B^{3n}) = H^0(B, \omega_B^{3n}) = 0\]
as $\omega_B^{3n}$ is anti-ample (recall that $B$ is Fano). As such, if $s$ denotes a global section of $\omega_Z^{-n}$, then the monomial $y^{3n}$ does not appear in an expression for $s$, and so the point $(0: 1: 0) \in Z_P$ annihilates $s$ for every $P \in B$. This shows that the codimension $2$ subvariety of $Z$ defined by $x = z = 0$? is contained in the base locus of $|\omega_Z^{-n}|$ for every positive $n$. As a consequence, $\omega_B^{-n}$ is not globally generated, hence $\omega_Z^{-1}$ can not be lef.
\end{proof}

\begin{prop}
Let $B$ be a Fano variety, suppose $X\to B$ is an elliptic Calabi--Yau which satisfies LTP, and let $\widetilde{W}_X\to W$ be the associated crepant resolution of the Weierstrass model of $X$ and suppose the LTP map is obtained by a sequence of blowups along smooth centres. Then $\mathscr{O}(\widetilde{W}_X)$ is not lef.
\end{prop}

\begin{proof}
Let  $\mathscr{E}=\cO_B\oplus \omega_B^{-2} \oplus \omega_B^{-3}$, and suppose the LTP map $\widetilde{\mathbb{P}(\mathscr{E})}\overset{\text{LTP}}\longrightarrow \mathbb{P}(\mathscr{E})$ which yields the crpeant resolution $\widetilde{W}_X\to W$ is obtained by a sequence of blowups of $\mathbb{P}(\mathscr{E})$ along smooth centers. To prove the proposition, we will show that the proper transform of the natural section $x = z = 0$ of $W_X\to B$ is in the base locus of the linear system $|n\widetilde{W}_X|$ for every $n \geq 1$. This will be achieved by induction on the number $k$ of blowups composing $\widetilde{\mathbb{P}(\mathscr{E})}\overset{\text{LTP}}\longrightarrow \mathbb{P}(\mathscr{E})$, the case $k = 0$ being Proposition \ref{prop: -kz not lef}.

For the inductive step, let $Z=\IP(\cE)$, and denote
\[\widetilde{Z} \xrightarrow{b} \hat{Z} \rightarrow Z,
\]
where $b$ is the $k$-th and last blow up, and let $\hat{W}$ be the proper transform of $W$ in $\hat{Z}$. Assume that $b$ is the blow up the smooth subvariety $C$ of $\hat{Z}$, defined by the ideal $\cI_C$, and denote by $E$ the exceptional divisor. Since $\widetilde{W}_X \rightarrow W_X$ is crepant, we have that $b(E)$ must be contained in the singular locus of $\hat{W}$. We then have
\[b^* \cO_{\hat{Z}}(n \hat{W}) = \cO_{\widetilde{Z}}(n \widetilde{W}_X) \otimes \cO_{\widetilde{Z}}(tE),\]
for some $t \in \IZ$. As a consequence
\[H^0(\tilde{Z}, \cO_{\widetilde{Z}}(n \widetilde{W}_X)) = H^0(\hat{Z}, \cO_{\hat{Z}}(n \hat{W}) \otimes \cI_C^t),\]
i.e., sections of $\cO_{\widetilde{Z}}(n \widetilde{W})$ can be identified with sections of $\cO_{\hat{Z}}(n \hat{W})$ which vanish along $C$ at least of order $t$. By the inductive hypothesis the proper transform of the natural section of $W$ is a component of the base locus of $\cO_{\hat{Z}}(n \hat{W})$, which is disjoint from the center $C$ of the blow up. As such, the proper transform of the section $x = z = 0$ of $W_X\to B$ is still in the base locus of $\cO_{\widetilde{Z}}(n \widetilde{W}_X)$, thus $\widetilde{W}_X$ cannot be lef.
\end{proof}

\section{LTP for crepant resolutions of Weierstrass 3-folds}\label{LTP3F}

Let $B$ be an arbitrary rational surface. We now consider 14 families of elliptic 3-folds $X\to B$ with simple gauge groups and verify that they satisfy LTP. Every elliptic Calabi--Yau $X\to B$ we consider in this section is of the form $X=\widetilde{W}$, where $\widetilde{W}\to W$ is a crepant resolution of a singular Weierstrass fibration $W\to B$. In each case, it turns out that $W\to B$ is in fact the Weierstass model of $\widetilde{W}\to B$, so that $W=W_{\widetilde{W}}$. 

\subsection{The Weierstrass fibrations under consideration} 
Given an anti-canonical Weierstrass fibration 
\[
W:(y^2z=x^3+Fxz^2+Gz^3)\subset \mathbb{P}(\cO_B \oplus \omega_B^{-2} \oplus \omega_B^{-3}),
\]
one may make a linear change of coordinates to put the fibration in \emph{Tate form}, which is given by
\begin{equation}\label{tf}
W:(y^2z + a_1xyz + a_3yz^2 = x^3 + a_2x^2z + a_4xz^2 + a_6z^3)\subset \mathbb{P}(\cO_B \oplus \omega_B^{-2} \oplus \omega_B^{-3}).
\end{equation}
In the Tate form each $a_i$ is a regular section of $\omega_B^{-i}$, and they are related to $F$ and $G$ by the equations
\[
F=\frac{-1}{48}(b_2^2-24b_4), \quad 
G=\frac{-1}{864}(36b_2b_4 - b_2^3 - 216b_6),
\]
where
\[
b_2 = a_1^2 + 4a_2, \quad
b_4 = a_1a_3 + 2a_4, \quad \text{and} \quad 
b_6 = a_3^2 + 4a_6.
\]
One may then employ Tate's algorithm \cite{Tate}\cite{KMSN1} to prescribe that the coefficients $a_i$ vanish to certain orders along a divisor $S\subset B$ in such a way that a particular singular fiber $\mathfrak{f}_W$ will appear over $S$ upon a resolution of singularities $\widetilde{W}\to W$. The dual graph of $\mathfrak{f}_W$ is then an affine Dynkin diagram associated with a Lie algebra $\mathfrak{g}$. The gauge group $\mathcal{G}_W$ associated with $W$ is then given by
\[
\mathcal{G}_W=\frac{\text{exp}(\mathfrak{g}^{\vee})}{\text{MW}_{\text{tor}}(\psi)}\times U(1)^{\text{rkMW}(\psi)},
\]  
where $\psi:W\to B$ is the associated projection to $B$, and $\text{MW}(\psi)$ denotes the Mordell--Weil group of rational sections of $\psi$. We note that it is possible for two distinct Weierstrass fibrations $W\to B$ and $W'\to B$ with distinct $\mathfrak{f}_W$ and $\mathfrak{f}_{W'}$ to give rise to the same gauge group, so that it is not necessarily the case that $\mathcal{G}_W\neq \mathcal{G}_{W'}$. 

Now let $S\subset B$ be a smooth divisor in the rational surface $B$. The equations of the Weierstrass fibrations we consider are all in Tate form as given by \eqref{tf}. We consider 14 distinct families of singular Weierstrass fibrations, whose explicit equations are given in Table~\ref{t1} along with the associated gauge groups. In each case, $x$ is a section of $\cO(1)\otimes \omega_B^{-2}$, $y$ is a section of $\cO(1)\otimes \omega_B^{-3}$, $z$ is a section of $\cO(1)$, $a_i$ is a section of $\omega_B^{-i}$ and $a_{i,k}$ is a section of $\omega_B^{-i}\otimes \cO(-kS)$, so that $W$ corresponds to the zero-locus of a section of $\cO(3)\otimes \omega_B^{-6}$. The total space of each fibration $W\to B$ is an anti-canonical divisor in $\mathbb{P}(\cO_B \oplus \omega_B^{-2} \oplus \omega_B^{-3})$ and singular along $x = y = s = 0$, where $s$ is the section of $\cO(S)$ whose zero-locus is the smooth divisor $S$. We take the map $b\mapsto (0:1:0)$ to be the distinguished section of each fibration.

\begin{table}[hbt]
\begin{center}
\begin{tabular}{|c|c|c|c|}
\hline
$W$ & $\mathcal{G}_W$ \\
\hline
$y^2z=x^3 + a_{4,1}sxz^2 + a_{6,2}s^2z^3$ & SU(2) \\
\hline
$y^2z + a_{3,1}syz^2=x^3 + a_{4,2}s^2xz^2 + a_{6,3}s^3z^3$ & SU(3) \\
\hline
$y^2z + a_1xyz=x^3 + a_{2,1}sx^2z + a_{4,2}s^2xz^2 + a_{6,4}s^{4}z^3$ & SU(4) \\
\hline
$y^2z + a_1xyz + a_{3,2}s^2yz^2=x^3 + a_{2,1}sx^2z + a_{4,3}s^{3}xz^2 + a_{6,5}s^{5}z^3$ & SU(5) \\
\hline
$y^2z=x^3 + a_2x^2z + a_{4,3}s^{3}xz^2 + a_{6,5}s^{5}z^3$ & USp(4) \\
\hline
 $y^2z=x^3 + a_2x^2z + sxz^2$ & SO(3) \\
\hline
$y^2z=x^3 + a_2x^2z + s^2xz^2$ & SO(5) \\
\hline
 $y^2z + a_1xyz=x^3 + sx^2z + s^2xz^2$ & SO(6) \\
\hline
 $y^2z=x^3 + a_{2,1}sx^2z + a_{4,2}s^2xz^2 + a_{6,4}s^4z^3$ & Spin(7) \\
\hline
$y^2z=x^3 + a_{4,2}s^2xz^2 + a_{6,3}s^3z^3$ & $G_2$ \\
\hline
$y^2z=x^3 + a_{4,3}s^3xz^2 + a_{6,4}s^4z^3$ & $F_4$ \\
\hline
$y^2z + a_{3,2}s^2yz^2=x^3 + a_{4,3}s^3xz^2 + a_{6,5}s^5z^3$ & $E_6$ \\
\hline
$y^2z=x^3 + a_{4,3}s^3xz^2 + a_{6,5}s^5z^3$ & $E_7$ \\
\hline
$y^2z=x^3 + a_{4,4}s^4xz^2 + a_{6,5}s^5z^3$ & $E_8$ \\
\hline
\end{tabular}
\caption{Equations for the anti-canonical Weierstrass fibrations under consideration.}
\label{t1}
\end{center}
\end{table} 

For each anti-canonical Weierstrass fibration $W$ listed in Table~\ref{t1}, a crepant resolution $\widetilde{W}\to W$ was constructed in \cite{Esole} by blowing up the projective bundle $\mathbb{P}(\cO_B \oplus \omega_B^{-2} \oplus \omega_B^{-3})$ along smooth complete intersections and then taking the proper transform $\widetilde{W}$ of $W$ along the blowups. In each case, the initial blowup is along the singular locus $x=y=s=0$, thus all divisors introduced in the resolution process do not meet the distinguished section. It then follows that in each example we have $W=W_X$, where $X=\widetilde{W}$ is the crepant resolution of $W$. For example in the SU(3), $G_2$ and USp(4) cases, the crepant resolution is obtained by two blowups. The first blowup $Z_1\to Z_0=\mathbb{P}(\cO_B \oplus \omega_B^{-2} \oplus \omega_B^{-3})$ is along its singular locus $\{x=y=s=0\}\subset Z_0$ with exceptional divisor $E_1$, and the second blow up is along $\{y = e_1 = 0\}\subset Z_1$, where $e_1 = 0$ is an equation for $E_1$ and $y$ denotes the pullback of the section $y$ under $Z_1\to Z_0$. We then summarize the resolution procedure  with the notation (as introduced in \cite{Esole})
\[
((x,y,s),(y,e_1)),
\]
where the first entry $(x,y,s)$ denotes the ideal along which the first blow up takes place, and the the second entry $(y,e_1)$ denotes the ideal along which the second blow up takes place. Such notation will then be used to summarize each resolution we consider, and all such resolution procedures are listed in Table~\ref{T2}. We note that if $Z_i\to Z_{i-1}$ denotes the $i$-th blow up with exceptional divisor $E_i$, then $e_i$ denotes a section of $\mathscr{O}(E_i)$ (we also elide the difference in notation between a section of a line bundle on $Z_{i-1}$ and its pullback via $Z_i\to Z_{i-1}$). 

\begin{table}[hbt]
\begin{center}
\begin{tabular}{|c|c|c|c|}
\hline
Resolution & $\mathcal{G}_Y$ \\
\hline
$(x,y,s)$ & SU(2) \\
\hline
$((x,y,s),(y,e_1))$ & SU(3) \\
\hline
$((x,y,s),(y,e_1),(x,e_2))$ & SU(4) \\
\hline
$((x,y,s),(x,y,e_1),(y,e_1),(y,e_2))$ & SU(5) \\
\hline
$((x,y,s),(y,e_1))$ & USp(4) \\
\hline
 $(x,y,s)$ & SO(3) \\
\hline
$((x,y,s),(x,y,e_1))$ & SO(5) \\
\hline
 $((x,y,s),(y,e_1),(x,e_2))$ & SO(6) \\
\hline
$((x,y,s),(y,e_1),(x,e_2))$ & Spin(7) \\
\hline
$((x,y,s),(y,e_1))$ & $G_2$ \\
\hline
$((x,y,s),(y,e_1),(x,e_2),(e_2,e_3))$ & $F_4$ \\
\hline
$((x,y,s),(y,e_1),(x,e_2),(e_2,e_3),(y,e_3),(y,e_4))$ & $E_6$ \\
\hline
$((x,y,s),(y,e_1),(x,e_2),(y,e_3),(e_2,e_3),(e_2,e_4),(e_4,e_5))$ & $E_7$ \\
\hline
$((x,y,s),(y,e_1),(x,e_2),(y,e_3),(e_2,e_3),(e_4,e_5),(e_2,e_4,e_6),(e_4,e_7)) $ & $E_8$ \\
\hline
\end{tabular}
\caption{Resolution procedure for each Weierstrass model.}
\label{T2}
\end{center}
\end{table}

\subsection{Verification of LTP}\label{sect: SO(5)}

\begin{lemma}\label{LMW0}
Let $\widetilde{W}\to W$ be any of the crepant resolutions as given in Table~\ref{T2}, and let $\pi:\widetilde{W}\to B$ be the structure map. Then $\emph{rank}\left(\emph{MW}(\pi)\right)=0$, and $W\to B$ is the Weierstrass model of $\pi:\widetilde{W}\to B$.
\end{lemma}

\begin{proof}
The fact that $\text{rank}\left(\text{MW}(\pi)\right)=0$ is standard (see, e.g., \cite{KLOOS}\cite{Esole}). And since each blowup in every resolution procedure given in Table~\ref{T2} introduces a fibral divisor which does not meet the distinguished section of $\widetilde{W}\to B$ (which we recall is given by $b\mapsto (0:1:0)$ for all $b\in B$), it follows that the crepant resolution $\pi:\widetilde{W}\to W$ contracts all divisors not meeting the section of $\widetilde{W}\to B$, thus $W\to B$ is the Weierstrass model of $\widetilde{W}\to B$, as desired.
\end{proof}

\begin{theorem}\label{LTP3F917}
Let $\widetilde{W}\to W$ be any of the crepant resolutions as given in Table~\ref{T2}, and let $\pi:\widetilde{W}\to B$ be the structure map. Then $\pi:\widetilde{W}\to B$ satisfies LTP. 
\end{theorem}
\begin{proof}
Since $\widetilde{W}$ is a 3-fold and $W\to B$ is the Weierstrass model of $\pi:\widetilde{W}\to B$ by Lemma~\ref{LMW0}, it follows from Proposition~\ref{PNX1971} that $\pi:\widetilde{W}\to B$ satisfies LTP if 
\[
n=\Gamma+\text{rank}\left(\text{MW}(\pi)\right),
\]
where $n$ is the number of blowups in the resolution procedure which yields $\widetilde{W}\to W$ as given in Table~\ref{T2}, and $\Gamma$ is the number of irreducible and reduced fibral divisors of $\widetilde{W}\to B$ not meeting the distinguished section. And since $W\to B$ does not contain any fibral divisors not meeting its distinguished section, and $\text{rank}\left(\text{MW}(\pi)\right)=0$ by Lemma~\ref{LMW0}, $\widetilde{W}\to B$ satisfies LTP if each blowup in the resolution procedure $\widetilde{W}\to W$ introduces a fibral divisor not meeting the distinguished section of $\pi:\widetilde{W}\to B$, which we now show.

Indeed, in each resolution procedure given in Table~\ref{T2}, the first blowup is along $(x,y,s)$, which does not meet the distinguished section, which we recall is given by $b\mapsto (0:1:0)$ for all $b\in B$. As such, the exceptional divisor $E_1$ of the first blowup is a fibral divisor which does not meet the distinguished section. The exceptional divisors of the subsequent blowups are also fibral divisors which do not meet the distinguished section, as they all correspond to blowups along some subvariety of the previous exceptional divisor, as can be seen from Table ~\ref{T2}. It then follows that in each resolution procedure given in Table~\ref{T2}, every blowup introduces a fibral divisor not meeting the distinguished section of $\pi:\widetilde{W}\to B$, thus $\widetilde{W}$ satisfies LTP.   
\end{proof}

\section{LTP for the \texorpdfstring{$E_6$}{E6}, \texorpdfstring{$E_7$}{E7} and \texorpdfstring{$E_8$}{E8} families}\label{LTP4F}

In this section, we consider 3 families of elliptic  Calabi--Yaus which are commonly referred to as the $E_8$, $E_7$ and $E_6$ families\footnote{We note that the $E_8$, $E_7$ and $E_6$ families defined here do \emph{not} coincide with the singular anti-canonical Weierstrass fibrations introduced in \S\ref{LTP3F} whose associated gauge groups are $E_8$, $E_7$ or $E_6$ as listed in Table~\ref{t1}.} (see, e.g., \cite{AE2} \cite{CCG}\cite{FvHH}\cite{F}\cite{KLRY}). We then show that when the base of such fibrations is of dimension 2 or 3, the associated elliptic Calabi--Yau 3- and 4-folds satisfy LTP (in the 4-fold case we require that the base is a toric Fano 3-fold), i.e., we construct LTP maps $\widetilde{\mathbb{P}(\mathscr{E})}\overset{\text{LTP}}\longrightarrow \mathbb{P}(\mathscr{E})$ and show that the upper Hodge diamonds of such resolutions coincide the with upper Hodge diamonds of the (blown up) projective bundles in which they are naturally embedded. We note that while the $E_6$, $E_7$ and $E_8$ families are often defined in such a way that they are not necesarily Calabi--Yau, for our purposes we will restrict to the Calabi--Yau case.

\subsection{The fibrations under consideration}\label{FUC71}

Let $B$ be a smooth compact complex projective variety. For $\text{dim}(B)=2$ we assume $B$ is a rational surface and for $\text{dim}(B)>2$ we assume $B$ is a Fano variety.

\begin{definition}[$E_8$ fibrations]
\label{DE8}
Let $\cE=\omega_B^{-2}\oplus \omega_B^{-3}\oplus \cO_B$.  An elliptic fibration $W\to B$ is said be an \emph{$E_8$ fibration} (or a \emph{smooth Weierstrass fibration}), if and only if $W$ a hypersurface in $\mathbb{P}(\cE)$ given by the equation
\begin{equation}\label{E8}
W: (y^2z=x^3+fxz^2+gz^3)\subset \mathbb{P}(\cE).
\end{equation}
The distinguished section is the constant section $s(b) = (0: 1: 0) \in W_b$.
\end{definition}
In the above equation, $x$, $y$ and $z$ denote regular sections of $\cO(1)\otimes \omega_B^{-2}$, $\cO(1)\otimes \omega_B^{-3}$ and $\cO(1)$ respectively, while $f$ and $g$ are sections of $\omega_B^{-4}$ and $\omega_B^{-6}$ respectively. To ensure that $W$ is smooth, we make the assumption that the hypersurfaces in $B$ given by $f=0$ and $g=0$ are both smooth and intersect transversally, or in other words, that $\{f=0\}\cup \{g=0\}$ is a normal crossing divisor with smooth irreducible components.

Fibrations of type $E_7$ were originally defined as a hypersurface in a weighted projective bundle $\mathbb{P}_{1,1,2}(\mathscr{E})$ over $B$. Here we recall the equivalent definition given in \cite[$\S$1]{CCG} as a hypersurface in a non-weighted $\IP^2$-bundle over $B$.

\begin{definition}[$E_7$ fibrations]
\label{DE7}
Let $\mathscr{E}=\cO_B\oplus \omega_B^{-1}\oplus \omega_B^{-2}$. An elliptic fibration $X\to B$ is said to be an \emph{$E_7$ fibration} if and only if $X$ is a hypersurface in $\mathbb{P}(\cE)$ given by
\begin{equation}\label{E7}
X:(y^2z-2x^2y+c_2x^2z+c_3xz^2+c_{4}z^3=0)\subset \mathbb{P}(\cE).
\end{equation}
The distinguished section is the constant section $s(b) = (0: 1: 0) \in X_b$.
\end{definition}

In the above equation, $x$, $y$ and $z$ are regular sections of $\cO(1)\otimes \omega_B^{-1}$, $\cO(1)\otimes \omega_B^{-2}$ and $\cO(1)$ respectively, and $c_i$ is a regular section of $\omega_B^{-i}$. To ensure $X$ is smooth we assume $\bigcup_{i=2}^{4}\{c_i=0\}$ is a normal crossing divisor with smooth irreducible components.

\begin{definition}[$E_6$ fibrations]
\label{DE6}
Let $\cE=\cO_B\oplus \omega_B^{-1}\oplus \omega_B^{-1}$. an elliptic fibration $Y\to B$ is said to be an \emph{$E_6$ fibration} if and only if $Y$ is a hypersurface in $\mathbb{P}(\cE)$ given by
\begin{equation}\label{E6}
Y:(x^3+y^3=b_1xyz+b_2xz^2+e_2yz^2+b_3z^3)\subset \mathbb{P}(\cE).
\end{equation}
The distinguished section is the constant section $s(b) = (0: 1: 0) \in Y_b$.
\end{definition}
In the above equation, $x$, $y$ and $z$ are regular sections of $\cO(1)\otimes \omega_B^{-1}$, $\cO(1)\otimes \omega_B^{-1}$ and $\cO(1)$ respectively, $b_i$ is a regular section of $\omega_B^{-i}$ and $e_2$ is a section of $\omega_B^{-2}$. To ensure $Y$ is smooth we assume $\bigcup_{i=1}^3\{b_i=0\}\cup \{e_2=0\}$ is a normal crossing divisor with smooth irreducible components.

\begin{prop}
$E_8$, $E_7$ and $E_6$ fibrations are all elliptic Calabi--Yaus.
\end{prop}

\begin{proof}
The proposition follows immediately from the adjunction formula together with a straightforward calculation of the arithmetic genus. For details, see e.g. \cite{AE2} or \cite{KLRY}.
\end{proof}

\subsection{LTP in the \texorpdfstring{$E_8$}{E8} case}\label{LTPE8}

Let $W\to B$ be an $E_8$ fibration. Recall that fibrations of type $E_8$ are in fact smooth Weierstrass fibrations, all of which admit the section $\sigma_0: B \rightarrow W$ given by $\sigma_0(b) = (0:1:0) \in W_b$, which we take to be the distinguished section of the fibration. Note that since $W\to B$ is its own Weierstrass model, an LTP map for an $E_8$ fibration is simply the identity.

\begin{theorem}\label{prop: LTPE8}
Let $W\to B$ be a generic 3- or 4-fold $E_8$ fibration as given by \eqref{E8}. In the 4-fold case, assume further that $B$ is a Fano toric 3-fold. Then $W$ satisfies LTP.
\end{theorem}
\begin{proof}
It follows from Remark~\ref{RM1X} that the relevant Hodge numbers for verifying LTP are $h^{1, 1}$ and, in the $4$-fold case, also $h^{1, 2}$. Since the Mordell--Weil rank of a generic smooth Weierstrass fibration is $0$ and the singular fibers of $W\to B$ are irreducible, for $\dim(B) > 1$ it follows from Proposition~\ref{prop: STW for CY} that
\[
h^{1,1}(W)=h^{1,1}(B)+1.
\]
Moreover, from Lemma \ref{lemma: hdg bundle}\eqref{pbf} we have
\[
h^{1,1}(\mathbb{P}(\cE))=h^{1,1}(B)+1,
\]
thus $h^{1,1}(W)=h^{1,1}(\mathbb{P}(\cE))$ for $W$ of arbitrary dimension. In particular, $W$ satisfies LTP when $W$ is a 3-fold, i.e., when $B$ is a rational surface. 

For the 4-fold case, we also need to consider $h^{1,2}$. By assumption, $B$ is a toric Fano 3-fold (there are 18 of them), and in such a case, one may use toric methods to compute $h^{1,2}(W)$. This was done in \cite[$\S$5, Table 1]{KLRY}, and the result is that $h^{1,2}(W) = 0$ in all cases. Since $h^{p, q}(B) = 0$ if $p \neq q$ by \cite[Theorem 9.3.2 or Theorem 9.4.9]{CLS}, it follows from item \eqref{pbf} of Lemma \ref{lemma: hdg bundle} that $h^{1, 2}(\IP(\cE)) = 0$. Hence $W$ satisfies LTP.
\end{proof}

\subsection{LTP in the \texorpdfstring{$E_7$}{E7} case}\label{LTPE7}

Let $X\to B$ be an $E_7$ fibration, which has two natural sections: $\sigma_0:B\hookrightarrow X$ and $\sigma_1:B\hookrightarrow X$ given by
\[
\sigma_0(b)=(0:1:0)\in \pi^{-1}(b) \quad \text{and} \quad \sigma_1(b)=(1:0:0)\in \pi^{-1}(b).
\]
We take $\sigma_0$ to be the distinguished section of the fibration, while $\sigma_1$ is a generator of $\MW(X)$, which is generically of rank $1$.

\begin{lemma}\label{LH11E7}
Let $X\to B$ be a generic $E_7$ fibration. Then $h^{1,1}(X)=h^{1,1}(B)+2$.
\end{lemma}
\begin{proof}
Over a generic point of the discriminant an $E_7$ fibration the singular fiber is a nodal cubic, which enhances to a cuspidal cubic or splits in two rational curves in higher codimension in the discriminant. As a consequence there is no irreducible and reduced fibral divisor which does not meet the section. Moreover, since the Mordell--Weil rank of an $E_7$ fibration is generically 1, it follows from the Shioda--Tate--Wazir formula for elliptic Calabi--Yaus \eqref{STWCY} that
\[
h^{1,1}(X)=h^{1,1}(B)+2,
\]
as desired.
\end{proof}

For the verification of LTP, we now construct the LTP map $\widetilde{\mathbb{P}(\mathscr{E})}\overset{\text{LTP}}\longrightarrow \mathbb{P}(\mathscr{E})$ which yields a crepant resolution of the Weierstrass model of $X$, which is given by
\[
W_X:\left(t^2 u = s^3 - \left( \frac{1}{3} c_2^2 -4c_4\right) s u^2 + \left( c_3^2 - \frac{2}{27} c_2^3 - \frac{8}{3} c_2c_4 \right) u^3\right)\subset \IP(\cE).
\]
In the Weierstrass model, the section $\sigma_0$ corresponds to the section $s_0(b) = (0: 1: 0)$ while $\sigma_1$ corresponds to $s_1(b) = (2c_2: -3c_3: 3)$. To construct a crepant resolution $\widetilde{W}_X\to W_X$, we blow up $\IP(\cE)$ along $s_1(B)$, which is given by the equations
\[
s_1(B):(t +c_3u = 3s - 2c_2u = 0)\subset \IP(\cE).
\] 
Taking the proper transform of $W_X$ then yields a resolution $\widetilde{W_X}\to W_X$. This resolution is crepant since $s_1(B)$ is a divisor in $W_X$ passing through the singular points of $W_X$, so the blow up only modifies the variety by introducing a $\IP^1$ over each singular point.

\begin{lemma}\label{LPB771}
Let $\widetilde{\IP(\cE)}$ denote the blow up of $\IP(\cE)$ along $\sigma_1(B)$. Then
\[
h^{p, q}(\widetilde{\IP(\cE)}) = h^{p, q}(B) + 2 h^{p - 1, q - 1}(B) + h^{p - 2, q - 2}(B).
\]
\end{lemma}
\begin{proof}
By formula \eqref{blowuphdp} we have
\[
E_{\widetilde{\IP(\cE)}}(u,v)=E_{\IP(\cE)}(u,v)+uvE_B(u,v),
\]
and since $E_{\IP(\cE)}(u,v)=(1+uv+(uv)^2)E_B(u,v)$, we then have
\[
E_{\widetilde{\IP(\cE)}}(u,v)=(1+2uv+(uv)^2)E_B(u,v),
\]
from which the lemma follows.
\end{proof}

\begin{theorem}\label{LTPE747}
Let $X\to B$ be a generic 3- or 4-fold $E_7$ fibration. In the 4-fold case, assume further that $B$ is a Fano toric 3-fold. Then $X$ satisfies LTP.
\end{theorem}
\begin{proof}
Let $\widetilde{\mathbb{P}(\mathscr{E})}\overset{\text{LTP}}\longrightarrow \mathbb{P}(\mathscr{E})$ be the LTP map correpsonding to the crepant resolution of $W_X$ constructed above. By Lemma~\ref{LPB771}, 
\[
h^{1, 1}(\widetilde{\IP(\cE)})= h^{1, 1}(B) + 2 h^{0, 0}(B)=h^{1, 1}(B) + 2,
\]
and 
\[
h^{1, 2}(\widetilde{\IP(\cE)})= h^{1, 2}(B) + 2 h^{0, 1}(B).
\]
By Lemma~\ref{LH11E7} it then immediately follows that $h^{1, 1}(\widetilde{\IP(\cE)})=h^{1,1}(X)$ in both the 3- and 4- fold cases. Thus $X$ satisfies LTP in the 3-fold case, as $h^{1, 1}$ is the only relevant Hodge number (cf. Remark~\ref{RM1X}). For the 4-fold case, the fact that $h^{1,2}(B)=h^{0,1}(B)=0$ for all Fano toric 3-folds (see \cite[Theorem 9.3.2 or Theorem 9.4.9]{CLS}) implies $h^{1, 2}(\widetilde{\IP(\cE)})=0$. Since $h^{1, 2}(X) = 0$ by \cite[$\S$5, Table 1]{KLRY}, $X$ satisfies LTP in the 4-fold case as well.
\end{proof}

\subsection{LTP in the \texorpdfstring{$E_6$}{E6} case}\label{LTPE6}

Let $Y\to B$ be an $E_6$ fibration, which has three natural sections $\sigma_i: B\to Y$ (which are also sections of $W_Y$) given by
\[
\sigma_i(b) = (-\zeta_3^i: 1: 0)\in \pi^{-1}(b), \qquad i = 0, 1, 2,
\]
where $\zeta_3=e^{\frac{2}{3}\pi \sqrt{-1}}$. We take $\sigma_0$ to be the distinguished section.

\begin{lemma}\label{LH11E6}
Let $Y\to B$ be a generic $E_6$ fibration. Then $h^{1,1}(Y)=h^{1,1}(B)+3$.
\end{lemma}
\begin{proof}
The reducible fibers of $E_6$ fibrations appear over loci of codimension greater than one in $B$, thus singular fibers of $E_6$ fibrations do not contribute to $h^{1,1}$. Since the Mordell--Weil rank of $E_6$ fibrations is 2, the Shioda--Tate--Wazir formula for elliptic Calabi--Yaus \eqref{STWCY} yields
\begin{equation}\label{H11E6}
h^{1,1}(Y)=h^{1,1}(B)+3.
\end{equation}
as desired.
\end{proof}

For the verification of LTP, we now construct the LTP map $\widetilde{\mathbb{P}(\mathscr{E})}\overset{\text{LTP}}\longrightarrow \mathbb{P}(\mathscr{E})$ which yields a crepant resolution of the Weierstrass model of $Y$, which is given by
\[
W_Y:\left(t^2 u = s^3 - F s u^2 + G u^3 \right)\subset \IP(\cE),
\]
where
\[F = \frac{1}{48} b_1^4 - \frac{9}{2} b_1 b_3 + 3 b_2 e_2, \qquad G = \frac{1}{864} b_1^6 + \frac{5}{8} b_1^3 b_3 - \frac{3}{4} b_1^2 b_2 e_2 + b_2^3 - \frac{27}{4} b_3^2 + e_2^3.\]
Apart the obvious constant section $s_0: B \rightarrow W_Y$ defined by $s_0(b) = (0: 1: 0)$ which corresponds to $\sigma_0$, we have two other sections $s_1$ and $s_2$ of $W_Y$ corresponding to $\sigma_1$ and $\sigma_2$ respectively. Let
\[\begin{array}{l}
\alpha_1 = \frac{1}{4} b_1^3 + \zeta_3 b_2 + \zeta_3^2 e_2,\\
\alpha_2 = \frac{1}{4} b_1^3 + \zeta_3^2 b_2 + \zeta_3 e_2,\\
\beta_1 = \frac{1}{18} \zeta_3 (1 - \zeta_3) b_1^3 - \frac{1}{2} (1 - \zeta_3^2) b_1 b_2 + \frac{1}{2} (1 - \zeta_3) b_1 e_2 + \frac{3}{2} \zeta_3 (1 - \zeta_3) b_3,\\
\beta_2 = -\frac{1}{18} \zeta_3 (1 - \zeta_3) b_1^3 - \frac{1}{2} (1 - \zeta_3) b_1 b_2 + \frac{1}{2} (1 - \zeta_3^2) b_1 e_2 - \frac{3}{2} \zeta_3 (1 - \zeta_3) b_3,\\
\end{array}\]
then $s_1 = (\alpha_1: \beta_1: -1)$ and $s_2 = (\alpha_2: \beta_2: -1)$.

To investigate the nature of the singularities of $W_Y$, we first analyze the singular fiber structure of $E_6$ fibrations following closely \cite[$\S$2.2]{AE2}. Over the generic point of the discriminant locus the singular fiber of the fibration is a nodal cubic curve, and they do not contribute to the singularities in the Weierstrass model. These curves can degenerate either to cuspidal curves or to fibers of type $I_2$ given by the union of a conic and a line. The equations for the locus over which we find fibers of type $I_2$ are
\[\Delta_Q^\rho : \left\{ \begin{array}{l}
e_2 = \rho b_2\\
b_3 = \frac{1}{27} b_1(9 \rho^2 b_2 - b_1^2)
\end{array} \right. \qquad \rho^3 \in \{ 1, \zeta_3, \zeta_3^2 \},\]
over which the fibers of the fibration are given by
\[
I_2:\left(x + \rho y + \frac{1}{3} \rho^2 b_1 z\right) \left(x^2 - \rho xy - \frac{1}{3} \rho^2 b_1 xz + \rho^2 y^2 - \frac{1}{3} b_1 yz + \left( \frac{1}{9} \rho b_1^2 - b_2 \right) z^2\right) = 0,
\]
which is (generically) the union of a line and a conic. We then have that generically over $\Delta_Q^\rho$, the sections $\sigma_i$ meet the fiber in either the line or the conic according to Table \ref{tab: intersection table}.

\begin{table}
\begin{center}
\begin{tabular}{|c||c|c|c|}
\hline
 & $\rho = 1$ & $\rho = \zeta_3$ & $\rho = \zeta_3^2$\\
\hline
\hline
$\sigma_0$ & line & conic & conic\\
\hline
$\sigma_1$ & conic & line & conic\\
\hline
$\sigma_2$ & conic & conic & line\\
\hline
\end{tabular}
\caption{How the sections $\sigma_i$ meet the $I_2$ fibres in $E_6$ fibrations.}
\label{tab: intersection table}
\end{center}
\end{table}

The map to the Weierstrass model contracts the component of these fibers which does not intersect $\sigma_0(B)$, and the corresponding points are then singular points of the Weierstrass model. It is then possible to see that the singular locus of the Weierstrass model has three irreducible components (one for each choice of $\rho$), with equations
\[
\text{Sing}_\rho: \left\{ \begin{array}{l}
e_2 = \rho b_2\\
b_3 = \frac{1}{27} b_1 (9\rho^2 b_2 - b_1^2)\\
t = 0\\
s + b_1^2 (b_1^2 - 4 \rho^2 b_2) = 0
\end{array} \right. \qquad \rho^3 \in \{ 1, \zeta_3, \zeta_3^2 \}.\]
The singular locus of $W_Y$ is then of codimension $4$ in $\mathbb{P}(\mathscr{E})$, and the irreducible components $\text{Sing}_\rho$ are contained in the sections $\Sigma_1 = s_1(B)$ and $\Sigma_2 = s_2(B)$ according to Table \ref{tab: containment table}.

\begin{table}
\begin{center}
\begin{tabular}{|c||c|c|c|}
\hline
 & $\rho = 1$ & $\rho = \zeta_3$ & $\rho = \zeta_3^2$\\
\hline
\hline
$\Sigma_1$ & \checkmark & \checkmark & No\\
\hline
$\Sigma_2$ & \checkmark & No & \checkmark\\
\hline
\end{tabular}
\caption{Which component of $\text{Sing}_\rho$ is contained in $\Sigma_i$?}
\label{tab: containment table}
\end{center}
\end{table}

We can then resolve the singularities of $W_Y$ as follows. First, we choose a section between $\Sigma_1$ and $\Sigma_2$, say $\Sigma_1$, and blow up the ambient space along it. Then we consider the proper transform of the Weierstrass fibration and we observe that we have no more singular points over $\Delta_Q^1$ and $\Delta_Q^{\zeta_3}$. The remaining singular points then lie over $\Delta_Q^{\zeta_3^2}$, and are contained in the proper transform of $\Sigma_2$. This proper transform is no longer isomorphic to the base $B$, rather it is the blow up of $B$ with center in $b_2 - e_2 = b_1^3 - 9 b_1 b_2 + 27 b_3 = 0$. After blowing up the ambient space along this subvariety, the proper transform of the fibration is smooth. This resolution of the Weierstrass model is crepant: the first blow up is along a divisor in $W_Y$ passing through the singular points of $W_Y$, so the blow up does not change the variety in codimension $1$ and does not introduce new divisors, and similarly for the second blow up.

In summary, we start with the ambient space $Z_0=\IP(\cE)$, and blow up a codimension $2$ subvariety isomorphic to $B$. Then in the new ambient space $Z_1$, we blow up a codimension $2$ subvariety isomorphic to the blow up $\widetilde{B}\to B$ along a codimension $2$ subvariety $C$, resulting in the ambient space $Z_2$, which we denote by $\widetilde{\mathbb{P}(\mathscr{E})}$. Taking the proper transform of $W_Y$ then yields a crepant resolution $\widetilde{W}_Y\to W_Y$, which is embedded as a hypersurface in $\widetilde{\mathbb{P}(\mathscr{E})}$.

\begin{lemma}\label{HNE6971}
Let $\widetilde{\mathbb{P}(\mathscr{E})}$ be as above. Then
\[
h^{p, q}(\widetilde{\IP(\cE)})= h^{p, q}(B) + 3h^{p- 1 , q - 1}(B) + h^{p - 2, q - 2}(B) + h^{p - 2, q - 2}(C).
\]
\end{lemma}
\begin{proof}
This is an application of \eqref{fbf} and \eqref{blowuphdp}.
\end{proof}

\begin{theorem}\label{LTPE647}
Let $Y\to B$ be a generic 3- or 4-fold $E_6$ fibration. In the 4-fold case, assume further that $B$ is a Fano toric 3-fold. Then $Y$ satisfies LTP.
\end{theorem}
\begin{proof}
Let $\widetilde{\mathbb{P}(\mathscr{E})}\overset{\text{LTP}}\longrightarrow \mathbb{P}(\mathscr{E})$ be the LTP map corresponding to the crepant resolution of $W_Y$ described above. By Lemma~\ref{LH11E6} and Lemma~\ref{HNE6971} we have
\[
h^{1,1}(\widetilde{\IP(\cE)})=h^{1,1}(B)+3=h^{1,1}(Y).
\]
Thus $Y$ satisfies LTP in the 3-fold case (recall that by Remark~\ref{RM1X} the only relevant Hodge number for LTP in the 3-fold case is $h^{1, 1}$). For the 4-fold case, Lemma~\ref{HNE6971} yields
\[
h^{1,2}(\widetilde{\IP(\cE)})=h^{1,2}(B)+3h^{0,1}(B),
\]
and since by Lemma~\ref{lemma: hdg inequality} we have $h^{1,0}(B) = 0$, in the 4-fold case $E_6$ fibrations satisfy LTP if and only if $h^{1,2}(\widetilde{\IP(\cE)}) = h^{1,2}(B)$, where we recall $B$ is a toric Fano 3-fold. Toric methods (see \cite[Theorem 9.3.2 or Theorem 9.4.9]{CLS}) may then be used to show that $h^{1, 2}(B) = 0$. Since $h^{1, 2}(\widetilde{W_Y}) = 0$ by \cite[$\S$5, Table 1]{KLRY}, $Y$ satisfies LTP in the 4-fold case as well.
\end{proof}

\section{LTP for the Borcea--Voison 4-fold}
While the $E_8$, $E_7$ and $E_6$ families are all defined over a base $B$ of arbitrary dimension, the next elliptic fibration we consider is an explicit 4-fold construction, introduced in \cite{CGP}. 

\begin{definition}[The Borcea-Voison 4-fold]
\label{BV4F}
Let $S_1$ and $S_2$ be two K3 surfaces, and suppose
\begin{enumerate}
\item $S_1$ admits an elliptic fibration $\pi: S_1 \rightarrow \IP^1$;
\item $S_2$ is a double covering of a del Pezzo surface.
\end{enumerate}
Both surfaces admit a natural involution: the surface $S_1$ has the hyperelliptic involution $\iota_1$, while $S_2$ has the covering involution $\iota_2$. The elliptic Calabi--Yau 4-fold $Z\to \mathbb{P}^2\times \mathbb{P}^1$ corresponding to the crepant resolution of the singular quotient $(S_1 \times S_2) / (\iota_1 \times \iota_2)$ will be referred to as the \emph{Borcea--Voisin 4-fold}.
\end{definition} 

Let $Z\to \mathbb{P}^2\times \mathbb{P}^1$ be the Borcea--Voisin 4-fold as given by Definition~\ref{BV4F}. We consider the case where the elliptic fibration on $S_1$ has only nodes or cusps as singular fibers (say $n$ singular fibres of type $II$ and $24 - 2n$ of type $I_1$), while $S_2$ is the double cover of $\IP^2$ branched along a smooth sextic. For the verification of LTP we first need the following

\begin{lemma}\label{LBV}
Let $Z$ be the Borcea--Voisin 4-fold constructed from an elliptic K3 surface with $n$ singular fibers of type $II$ and $24 - 2n$ of type $I_1$ and a double cover of $\IP^2$ branched along a smooth sextic. Then $h^{1, 1}(X)=5$ and
$h^{1, 2}(X) = 30$.
\end{lemma}
\begin{proof}
To prove the lemma one needs first to study the fixed locus of the two natural involutions on $S_1$ and $S_2$, and then to follow the quotient and blow up needed to produce $X$ from $S_1 \times S_2$. We refer to \cite[Proposition 4.2]{CGP} for all the details, as the proof of the lemma follows from the proof of \cite[Proposition 4.2]{CGP} with the obvious adjustments.
\end{proof}

Let $\cE=\cO_{\IP^2\times \IP^1}\oplus \omega_{\IP^2\times \IP^1}^{-2}\oplus \omega_{\IP^2\times \IP^1}^{-3}$. For the verification of LTP, we now construct an LTP map $\widetilde{\mathbb{P}(\mathscr{E})}\overset{\text{LTP}}\longrightarrow \mathbb{P}(\mathscr{E})$ which yields a crepant resolution $\widetilde{W}_Z\to W_Z$ of the Weierstrass model $W_Z$ of the Borcea--Voisin 4-fold $Z\to \mathbb{P}^2\times \mathbb{P}^1$. The Weierstrass model of $Z$ is given by
\[
W_Z:(y^2z=x^3+(Af^2)xz^2+(Bf^3)z^3)\subset \mathbb{P}(\cE),
\]
where $f=0$ is the sextic curve $C\subset \IP^2$ which is the branch locus of $S_2\to \IP^2$, and $A$ and $B$ are such that the Weierstrass model of $S_1\to \IP^1$ is given by $s^2t=u^3+Aut^2+Bt^3$. For the sake of constructing a crepant resolution of $W_Z$ we make the assumption that $2\partial B/B\neq 3\partial A/A$.  

The Weierstrass model $W_Z$ is singular along the smooth surface $S\subset \mathbb{P}(\cE)$ given by $x=y=f=0$, and we now construct an explicit crepant resolution $\widetilde{W}_Z\to W_Z$. For this, we first blow up $\IP(\cE)$ along $S$, and we let $(X_1:X_2:F)$ denote the coordinates in the exceptional divisor of the blow up. In the chart $X_2= 1$, we have $ x = X_1y$ and $f = Fy$, so that the exceptional divisor is given by $y = 0$, and the proper transform of $W_Z$ is given by
\[
1=y(X_1^3+AF^2X_1+BF^3),
\]
thus the exceptional divisor and the proper transform are disjoint. In the chart $X_1= 1$, we have $ y = X_2x$ and $f = Fx$, thus the exceptional divisor is given by $x = 0$ and the proper transform of $W_Z$ is given by
\[
 X_2^2 = x(1 + AF^2 + BF^3). 
 \]
The proper transform of $W_Z$ is then singular along $x= X_2 = 1 + AF^2 + BF^3 = 0$. In the chart $F = 1$, we have $x = X_1f$ and $y = X_2f$, so that the exceptional divisor is given by $f = 0$ and the proper transform of $W_Z$ is given by
\[
X_2^2 = f(X_1^3 + AX_1 + B),
\]
which is singular along the surface
\[
T:(f =X_2= X_1^3 + AX_1 + B = 0)\subset Z_1,
\]
where $Z_1$ denotes the blow up of $\IP(\cE)$ along $S$. Observe that this description of the singular locus patches with the one in the previous chart. Since in the previous chart there is no singular point contained in $F = 0$, we see that the whole singular locus is described in this chart. The previous assumption that $2\partial B/B\neq 3\partial A/A$ assures that $T$ is in fact smooth. Moreover, the blow up of $Z_1$ along $T$ yields a crepant resolution of $\widetilde{W}_Z\to W_Z$ by taking the proper transform of $W_Z$ through the two blow ups. It is possible to see by direct computations that this resolution is crepant. By Lemma~\ref{LBV}, the Borcea--Voison 4-fold $Z\to \mathbb{P}^2\times \mathbb{P}^1$ then satisfies LTP if and only if $h^{1,1}(\widetilde{\IP(\cE)})=5$ and $h^{1,2}(\widetilde{\IP(\cE)})=30$, where $\widetilde{\IP(\cE)}$ denotes the blow up of $Z_1$ along $T$.

\begin{theorem}\label{LTPVB647}
The Borcea-Voison 4-fold $Z\to \mathbb{P}^2\times \mathbb{P}^1$ satisfies LTP.
\end{theorem}
\begin{proof}
Let $\widetilde{\mathbb{P}(\mathscr{E})}\overset{\text{LTP}}\longrightarrow \mathbb{P}(\mathscr{E})$ be the LTP map correpsonding to the crepant resolution of $W_Z$ constructed above. By formulas \eqref{h11pb} and \eqref{h12pb} we have
\[
h^{1,1}(\widetilde{\IP(\cE)})=h^{1,1}(\IP^2\times \IP^1)+3=5,
\]
and 
\[
h^{1,2}(\widetilde{\IP(\cE)})=h^{1,2}(\IP^2\times \IP^1)+h^{1,0}(S)+h^{1,0}(T)=h^{1,0}(S)+h^{1,0}(T),
\]
where the last equality follows from the fact that $h^{1,2}(\IP^2\times \IP^1)=0$. We then have a match at the level of $h^{1,1}$, thus the Borcea--Voisin 4-fold satisfies LTP if and only if $h^{1,0}(S)+h^{1,0}(T)=30$.

Now the surface $S$ is isomorphic to $C\times \IP^1$ (where we recall $C$ is a smooth sextic curve of in $\IP^2$), and since $C$ is of genus $h^{1,0}(C)=10$, it then follows from equation \eqref{fbf} that $h^{1,0}(S)=10$. As for the surface $T$, we note that this surface is isomorphic to the product of $C$ with the $3:1$ cover of $\IP^1$ with $n$ ramification points with multiplicity $3$ and $24 - 2n$ ramification points with multiplicity $2$, which by the Riemann--Hurwitz formula is a curve of genus 10. It then follows that $h^{1,0}(T)=10 +10=20$, which yields
\[
h^{1,2}(\widetilde{\IP(\cE)})=10+20=30,
\]
thus $Z$ satisfies LTP.
\end{proof}

\section{LTP and the Calabi--Yau condition}\label{CYCON}

In this section we show that if we drop the Calabi--Yau condition on the total space of the elliptic fibration, then the Lefschetz-type phenomenon can fail to hold.

Let $W \to \IP^1$ be a smooth elliptic $K3$ surface given by a Weierstrass equation
\begin{equation}\label{eq: k3 example}
W: (T^2 U = S^3 + \alpha S U^2 + \beta U^3) \subset \IP(\cO_{\IP^1} \oplus \cO_{\IP^1}(4) \oplus \cO_{\IP^1}(6)),
\end{equation}
where $\alpha \in H^0(\IP^1, \cO_{\IP^1}(8))$, $\beta \in H^0(\IP^1, \cO_{\IP^1}(12))$ are generic, and let $\pi:X\to \IP^1 \times \IP^n$ be the elliptic fibration whose total space $X$ is the cartesian product $W\times \IP^n$ for $n>0$. It then follows that $\pi:X\to \IP^1 \times \IP^n$ is a smooth Weierstrass fibration, whose equation is given by 
\[
X: (y^2 z = x^3 + \alpha x z^2 + \beta z^3) \subset Z,
\]
where $Z = \IP(\cO_{\IP^1 \times \IP^n}\oplus \cL^2 \oplus \cL^3)$ and $\cL=\cO_{\IP^1 \times \IP^n}(2, 0)$. From the adjunction formula one may deduce that $\omega_X$ is non-trivial, so that $X$ is not Calabi--Yau. As $X$ is a product and $Z$ is a $\IP^2$-bundle, we can easily compute the Hodge diamonds of $X$ and $Z$. In particular, at the level of $h^{1,1}$ we have
\[h^{1, 1}(X) = 21, \quad \text{and} \quad  h^{1, 1}(Z) = 3,
\]
thus $X$ does \emph{not} satisfy LTP.

\begin{rem}
Observe that even though the total space $X$ of the fibration $\pi$ is a product, $X$ is not birational over $\IP^1 \times \IP^n$ to an elliptic fibration of the form $\IP^1 \times \IP^n \times E \rightarrow \IP^1 \times \IP^n$ ($E$ is an elliptic curve). To see that this is true, assume that we have a diagram
\[\xymatrix{W \times \IP^n \ar@{-->}[rr]^f \ar[dr] & & \IP^1 \times \IP^n \times E \ar[dl]\\
 & \IP^1 \times \IP^n. & }\]
The indeterminacy loci of $f$ and of $f^{-1}$ are of codimension at least $2$, hence their image in $\IP^1 \times \IP^n$ is contained in a divisor. So for a generic point $P \in \IP^n$ the restriction of $f$ to $\IP^1 \times \{ P \}$ will exhibit $W$ as birational to a product, which is not the case.
\end{rem}

\section{LTP conjectures}\label{LTPCON}

We now formulate two conjectures, which we will refer to as the `LTP-weak conjecture' and the `LTP-strong conjecture'. A positive answer to the latter would imply a positive answer to the former.

\begin{conj}[LTP-weak conjecture]\label{C117}
An elliptic Calabi--Yau satisfies LTP.
\end{conj}

If Conjecture~\ref{C117} is in fact true, it is then natural to surmise that LTP for an elliptic Calabi--Yau $X\to B$ is a consequence of the inclusion map $i:\widetilde{W}_X\hookrightarrow \widetilde{\mathbb{P}(\mathscr{E})}$ associated with an LTP map $\widetilde{\mathbb{P}(\mathscr{E})}\overset{\text{LTP}}\longrightarrow \mathbb{P}(\mathscr{E})$ for $X$ inducing an isomorphism $i^*: H^k(\widetilde{\mathbb{P}(\mathscr{E})}, \mathbb{Z}) \rightarrow H^k(\widetilde{W}_X, \mathbb{Z})$ for $k < \dim X$. However, this more general statement may fail to hold, as we now show.

\begin{prop}\label{CE119}
Let $X\to B$ be an elliptic Calabi--Yau which satisfies LTP, and let $i:\widetilde{W}_X\hookrightarrow \widetilde{\mathbb{P}(\mathscr{E})}$ be the inclusion map associated with an LTP map $\widetilde{\mathbb{P}(\mathscr{E})}\overset{\emph{LTP}}\longrightarrow \mathbb{P}(\mathscr{E})$ for $X$. If the Mordell-Weil group of $X\to B$ admits torsion, then the map $i^*: H^2(\widetilde{\mathbb{P}(\mathscr{E})}, \mathbb{Z}) \rightarrow H^2(\widetilde{W}, \mathbb{Z})$ induced by the inclusion map $i$ is not an isomorphism.
\end{prop}

\begin{proof}
Since the Mordell-Weil group of $X\to B$ admits torsion, there exists a rational section $\tau$ of $X\to B$ of finite order. From the surjective group homomorphism $\NS(\widetilde{W}_X)\to \MW(\widetilde{W}_X)$ which associates a class $D$ with its restriction to the generic fiber of $X\to B$, we deduce that there exists in $\NS(\widetilde{W}_X)$ a class of finite order, namely $\overline{\tau(B)}$. As $\widetilde{W}_X$ is an elliptic Calabi--Yau, we have $\NS(\widetilde{W}_X) = H^2(\widetilde{W}_X, \IZ)$, so that $H^2(\widetilde{W}_X, \IZ)$ admits torsion. But $H^2(\widetilde{\mathbb{P}(\mathscr{E})}, \IZ)$ is torsion-free, thus it can not be isomorphic to $H^2(\widetilde{W}_X, \IZ)$.
\end{proof}

In light of Proposition~\ref{CE119}, perhaps LTP is a consequence of LTP maps inducing an isomrphism of \emph{rational} Hodge structures (in the appropriate dimensions). In particular, we make the following conjecture, which refer to as the `LTP-strong conjecture'.

\begin{conj}[LTP-strong conjecture]
Let $X\to B$ be an elliptic Calabi--Yau. Then the following statements hold.
\begin{enumerate}
    \item 
    An LTP map $\widetilde{\mathbb{P}(\mathscr{E})}\overset{\emph{LTP}}\longrightarrow \mathbb{P}(\mathscr{E})$ for $X$ exists.
    \item
    Given an LTP map $\widetilde{\mathbb{P}(\mathscr{E})}\overset{\emph{LTP}}\longrightarrow \mathbb{P}(\mathscr{E})$, the map $i^*: H^k(\widetilde{\mathbb{P}(\mathscr{E})}, \IQ) \rightarrow H^k(\widetilde{W}_X, \IQ)$ induced by the inclusion $i: \widetilde{W}_X \hookrightarrow \widetilde{\mathbb{P}(\mathscr{E})}$ associated with the LTP map is an isomorphism for $k < \dim X$.
\end{enumerate}
\end{conj}

\bibliographystyle{plain}
\bibliography{LTP_FTV}

\end{document}